\documentclass[11pt]{amsart}
\usepackage{amsmath,amsfonts,amssymb,upgreek, amscd, mathtools}
\usepackage[breaklinks]{hyperref}

\theoremstyle{thrm}
\usepackage{tikz}
\usepackage{enumerate}
\usetikzlibrary{calc,decorations.markings}
\usepackage{tikz-cd}
\usepackage{graphicx}

\theoremstyle{plain}

\newtheorem{thm}{Theorem}[section]

\newtheorem{prop}[thm]{Proposition}

\newtheorem{cor}[thm]{Corollary}

\newtheorem{question}[thm]{Question}
\newtheorem{defn}[thm]{Definition}
\newtheorem{example}[thm]{Example}
\newtheorem*{ack}{Acknowledgement}

\theoremstyle{definition}
\newtheorem{remark}[equation]{Remark}
\newcommand{\Ad}{\operatorname{Ad} }
\newcommand{\e}{\operatorname{e} }
\newcommand{\ad}{\operatorname{ad} }
\newcommand{\Der}{\operatorname{Der} }

\newcommand{\RBO}{\operatorname{RBO} }
\newcommand{\Sub}{\operatorname{Sub} }
\newcommand{\Exp}{\operatorname{Exp} }

\newcommand{\I}{\operatorname{I} }

\newcommand{\Aut}{\operatorname{Aut} }

\numberwithin{equation}{section}


\begin{document}
	
	\title{Rota--Baxter operators and skew left brace structures over Heisenberg Group}

	\author{Nishant Rathee}
	\address{Department of Mathematical Sciences, Indian Institute of Science Education and Research (IISER) Mohali, Sector 81, SAS Nagar, P O Manauli, Punjab 140306, India} 
	\email{nishantrathee@iisermohali.ac.in, monurathee2@gmail.com}

\begin{abstract}
Rota–Baxter operators on groups have been recently defined in \cite{LHY2021}, and they share a close connection with skew braces, as demonstrated in \cite{VV2022}. In this paper, we classify all Rota–Baxter operators of weight 1 on the Heisenberg Lie algebra of dimension 3 up to the Jordan canonical form of a \(2 \times 2\) matrix block by automorphisms of the Heisenberg Lie algebra. Using the fact that the exponential map from the Heisenberg Lie algebra to the Heisenberg group is bijective, we induce these operators on the Heisenberg group. Finally, we enumerate all skew left brace structures on the Heisenberg group induced by these Rota–Baxter operators.
\end{abstract}	
	
	\maketitle
	
	\section{Introduction}
	G. Baxter first introduced Rota--Baxter operators in 1960, within the context of fluctuation theory in probability, presenting them as a generalization of the integral operator concept in the algebra of continuous functions. Over the following two decades, these operators have gained significant prominence, playing crucial roles in various domains such as combinatorics, operad splitting, classical solutions of the Yang--Baxter equation, and applications in Poisson geometry and mathematical physics, see \cite{GCR1969}, \cite{PBG17}, \cite{UK08}. L. Guo, H. Lang, and Y. Sheng defined  Rota--Baxter operators of weight $1$  on Lie groups in \cite{LHY2021}. Their work involved the analysis of smooth Rota--Baxter operators on Lie groups and demonstrated that such operators are differentiable to Rota--Baxter operators of weight $1$ on the associated Lie algebras.  The concept of Rota--Baxter operators has been generalized by Jiang, Sheng, and Zhu \cite{JYC22} to include relative Rota--Baxter operators on Lie groups. Additionally, another generalization of Rota--Baxter operators to Clifford semigroups, with applications to set-theoretical solutions of the Yang--Baxter equation, has been explored in \cite{CMP}.
	
Skew left braces are algebraic structures that Vendramin and Guarnieri introduced in \cite{GV17}, and they share  a close connection with non-degenerate set-theoretical solutions of the Yang--Baxter equation. V.G. Bardakov and V. Gubarev explored Rota--Baxter operators over  abstract groups in \cite{VV2022}, \cite{VV2023}, where they established the link between Rota--Baxter groups and  skew left braces. They proved that every Rota--Baxter operator on a group induces a skew left brace. Moreover, they established that every skew left brace can be embedded into a Rota--Baxter group, laying the groundwork for inquiries into the specific skew left braces that Rota--Baxter operators can induce. This question was addressed by A. Caranti and L. Stefanello in \cite{AS22}. More specifically, they identified certain skew left braces for which, despite their associated 
	$\lambda$-map lying within the subgroup of inner automorphisms, cannot be derived from Rota--Baxter operators. 
	
While the classification of Rota--Baxter operators across algebras and Lie algebras has been thoroughly addressed, for example see \cite{JH18, PBG14, PBG17, ZGR19}. Moreover, an algorithm has been developed  in \cite{NM1} to classify Rota--Baxter operators over finite groups of small order. However, examples or classifications of Rota--Baxter operators over Lie groups are not widely known.   The following theorem for Rota--Baxter operators  closely resembles the well-known relationship between the homomorphisms of Lie groups and Lie algebras, that is, if there is a homomorphism at the level of Lie groups, then its tangent map at the identity induces a Lie algebra homomorphism.

\begin{thm}\cite[Theorem 2.10]{LHY2021}
Let \(G\) be a Lie group and \(\mathfrak{g} = T_{e}G\). If \(\mathcal{R}: G \rightarrow G\) is a Rota--Baxter operator, then  
\[ R = \mathcal{R_{*}}_{e}: \mathfrak{g} \rightarrow \mathfrak{g} \]
the tangent map of \(\mathcal{R}\) at the identity \(e\), defines a Rota--Baxter operator on the Lie algebra \(\mathfrak{g}\).
\end{thm}

The motivation of this paper is to explore the converse of this theorem, more precisely, inducing Rota--Baxter operators from a Lie algebra to its corresponding Lie group. This task presents significant challenges, even in the context of homomorphisms, and is possible only for simply connected groups. 
In this paper, we present the converse of the above theorem for the Heisenberg Lie group. Although a local Lie theory for simply connected relative Rota--Baxter groups is provided in \cite{JYC22}, it induces the relative Rota--Baxter operator only in some neighborhood of the identity. We will be using the same formula to induce Rota--Baxter operators on the Heisenberg group. This is achieved by first classifying all Rota--Baxter operators of weight $1$ over its Lie algebra, and then we use the properties of the associated exponential map to transfer these operators to the Heisenberg Group. A similar situation arises when examining skew left brace structures over Lie groups. Skew left brace structures over Lie groups are not widely known, which presents another layer of motivation for this study. This work aims to explore and list skew left brace structures over the Heisenberg group, focusing primarily on those induced by Rota--Baxter operators.

The paper is organized as follows: Section 2 introduces the necessary preliminaries about Rota--Baxter groups, Rota--Baxter Lie algebras, and skew left braces. Section 3 is dedicated to the exponential map of the semi-direct product of Lie groups, where we explicitly determine the exponential map for the Lie group \(G \ltimes_{\Ad} G\) (Theorem \ref{semiexp}). In Section 4, we identify all Rota--Baxter operators of weight 1 on the Heisenberg Lie algebra of dimension three. We obtain the matrices for these Rota--Baxter operators by directly solving the operators defining equations, using Mathematica \cite{math}. In Section 5, we transfer these Rota--Baxter operators to the Heisenberg Lie group (Theorem \ref{RBOHLG}) and ultimately enumerate all the skew left brace structures induced by these Rota--Baxter operators.

	\section{Preliminaries}
	
	In this section, we shall revisit basic results and observations concerning Rota--Baxter groups and their connection with left skew braces. Our main references are \cite{LHY2021,VV2022}.
	
	\subsection{Rota--Baxter groups and skew left braces}
	
	\begin{defn}
		Let $(G, \cdot)$ be a group. A set map $\mathcal{R}: G \rightarrow G$ is said to be Rota--Baxter operator  on the group $G$ if 
		\begin{align*}
			\mathcal{R}(x) \cdot \mathcal{R}(y)= \mathcal{R}(x \cdot \mathcal{R}(x) \cdot y \cdot \mathcal{R}(x)^{-1}), \text{for all } x, y \in G.
		\end{align*}
		A group $(G, \cdot)$ equipped with a Rota--Baxter operator  is called a Rota--Baxter group. We denote a Rota--Baxter group simply by the pair $(G,\mathcal{R})$.
	\end{defn}

\begin{defn}
Two Rota--Baxter operators	$\mathcal{R}$ and $\mathcal{B}$ on a group \(G\) are said to be equivalent if there exists an isomorphism \(\Psi\) of \(G\) such that \(\Psi \mathcal{R} \Psi^{-1} = \mathcal{B}\).
\end{defn}

\begin{example}\label{groupexample}
	Let $G$ be a group. Then the following are examples of Rota--Baxter operators
	\begin{itemize}
		\item[i)] The map $\mathcal{R}_e(x)=e$ is a Rota--Baxter operator on $G$, where $e$ is identity of $G$.
		
		\item[ii)] The map $\mathcal{R}_{-1}(x)=x^{-1}$ is a Rota--Baxter operator on $G$.
		
		\item[ii)] Given an exact factorization  $G=HL$, the map $\mathcal{R} : G \rightarrow G$, defined as  $\mathcal{R}(hl)=l^{-1} $ is a Rota--Baxter operator on $G$.
	\end{itemize}
\end{example}

Let $G$ be a group, and let the adjoint map $\Ad: G \rightarrow \Aut(G)$ be defined by $\Ad_x(y)=xyx^{-1}$ for all $x,y \in G$. Then, we can define the semi-direct product $G \ltimes_{\Ad} G$, where the group operation is defined as follows:
\begin{align*}
	(x_1, y_1) (x_2, y_2) := (x_1 x_2, y_1 \Ad_{x_1}(y_2)) = (x_1 x_2, y_1 x_1 y_2 x_1^{-1}),
\end{align*}
for all $x_1, x_2, y_1, y_2 \in G.$ 

Next, we provide a sketch of the proof regarding the relationship between subgroups of $G \ltimes_{\Ad} G$ and Rota--Baxter operators on $G$. For further details, we recommend that the reader refer to \cite[Proposition 2.13]{LHY2021} and \cite[Proposition 3.9]{NM1}.

\begin{prop}\label{grptorb}
Let $G$ be a group. Then Rota--Baxter operators on $G$ are in one-to-one correspondence with subgroups $H$ of $G \ltimes_{\Ad} G$ such that the projection of the second coordinate of $H$ onto $G$ is a bijection.
\end{prop}
\begin{proof}
Let $\Sub(G)$ denote the set of subgroups $H$ of $G \ltimes_{\Ad} G$ such that the projection of the second coordinate of $H$ onto $G$ is a bijection, and let $\RBO(G)$ denote the set of all Rota--Baxter operators on $G$. For $\mathcal{R} \in \RBO(G)$ and $x \in G$ , define the mapping 
$$\psi: \RBO(G) \rightarrow \Sub(G) \mbox{ by } \mathcal{R} \mapsto \text{Gr}(\mathcal{R}):=(\mathcal{R}(x), x).$$

Let $H \in \Sub(G)$ and $x \in G$. Utilizing the fact that the projection of the second coordinate of $H$ onto $G$ is a bijection, there exists a unique $y \in G$ such that $(y, x) \in H$.  We can define $\mathcal{R}_H: G \rightarrow G$ given by $\mathcal{R}_H(x)=y$. It is easy to see that $R_H$ constitutes a Rota--Baxter operator on $G$. Thus, we have  the map 
$$
\phi: \Sub(G) \rightarrow \RBO(G) \mbox{ defined by } \phi(H)=\mathcal{R}_H.$$
We can easily verify that $\psi$ and $\phi$ are well-defined and inverses of each other, which proves the result.
 \end{proof}

	\begin{defn}
		An algebraic structure $(A, + , \circ)$ is said to be a skew left brace if $(A,  +)$ and $(A, \circ)$ are groups and the following compatibility condition holds
		\begin{equation}\label{bcomp}
			a \circ (b  + c ) = a \circ b -a + a\circ c
		\end{equation}
		\noindent for all $ a, b , c \in A$,  where $-a$ denotes the inverse of $a$ with respect to $+$.  
	\end{defn}		

	\begin{example}
	All groups with same  `$+$' and `$\circ$' operations are examples of trivial skew left braces.
\end{example}

\begin{example}
	A Radical ring $(R, +, \cdot)$ with $a \circ b = a+b +a \cdot b$ is a skew left brace.
\end{example}

\begin{example}
	Let $n \geq 2$ , then $A_{(n,d)}:=(\mathbb{Z}_{n}, +, \circ)$ is a skew left brace, where $d$ is a divisor of $n$ such that every prime divisor of $n$ divides $d$, and $ a \circ b = a+b+dab.$
\end{example}

Next we state the relationship between skew left braces and Rota--Baxter groups.

\begin{thm}\label{VVSB}\cite[Theorem 3.1]{VV2022}
Let $(G, \cdot)$ be a group and $\mathcal{R} : G \rightarrow G$ be a Rota--Baxter operator. Then the operation 
$x \circ_{\mathcal{R}} y:=x \cdot \mathcal{R}(x)\cdot y\cdot \mathcal{R}(x)^{-1}$ defines a group structure on $G$  called descendant group of Rota--Baxter group  $(G, \mathcal{R})$ and $(G, \cdot, \circ_{\mathcal{R}})$ forms a skew left brace.
\end{thm}

	 \subsection{Rota--Baxter Lie algebra}
	 \begin{defn}
Let $(\mathfrak{g}, [\cdot, \cdot])$ be a Lie algebra. Then a linear map $R:  \mathfrak{g} \rightarrow \mathfrak{g}$ is called a Rota--Baxter operator if it satisfies 
$$[R(x), R(y)]=R\big([R(x), y]+ [x, R(y)]+ [x,y] \big),$$
for all $x, y \in \mathfrak{g}.$ A Lie algebra $\mathfrak{g}$ equipped with a Rota--Baxter operator  is called a Rota--Baxter Lie algebra. We denote a Rota--Baxter Lie algebra simply by the pair $(\mathfrak{g}, R)$.
 \end{defn}
	\begin{remark}
	 Note that the definitions of a Rota--Baxter operator provided for groups and Lie algebras are more precisely definitions of Rota--Baxter operators of weight $1$. However, throughout this article, we refer to a Rota--Baxter operator of weight $1$ simply as a Rota--Baxter operator on the corresponding group or Lie algebra.
	 \end{remark}
 
 \begin{defn}
 	Two Rota--Baxter operators	$R$ and $B$ on a group $\mathfrak{g}$ are said to be equivalent if there exists an isomorphism \(\psi\) of \(\mathfrak{g}\) such that \(\psi R \psi^{-1} =B\).
 \end{defn}

\begin{example}\label{lieexample}
	Let $\mathfrak{g}$ be a Lie algebra. Then the following are examples of Rota--Baxter operators
	\begin{itemize}
		\item[i)] The map $R(x)=0$ for all $x \in \mathfrak{g}$ is a trivial  Rota--Baxter operator on $\mathfrak{g}$. 
		
		\item[ii)] For $x \in   \mathfrak{g}$, the map $R_{-1}(x)=-x$ is a Rota--Baxter operator on $\mathfrak{g}$.
		
		\item[ii)] Let $\mathfrak{g}_{+}$ and $\mathfrak{g}_{-}$ be Lie subalgebras of $\mathfrak{g}$ such that $\mathfrak{g} = \mathfrak{g}_{+} \oplus \mathfrak{g}_{-}$. The map $B : \mathfrak{g} \rightarrow \mathfrak{g}$, defined as $B(x_{+} + x_{-}) = -x_{-}$, is a Rota--Baxter operator on $\mathfrak{g}.$
	\end{itemize}
\end{example}

\begin{thm}\cite[Theorem 2.10]{LHY2021}
Let $(G, \mathcal{R})$ be a Rota--Baxter Lie group. Let  $\mathfrak{g}= T_{e}G$ be the Lie algebra of $G$ and
$$ R = \mathcal{R_{*}}_{e}: \mathfrak{g} \rightarrow  \mathfrak{g}$$
the tangent map of $\mathcal{R}$ at the identity $e$. Then $(\mathfrak{g}, R)$ is a Rota--Baxter Lie algebra.
\end{thm}

\begin{remark}
 The Rota--Baxter operators in Example \ref{lieexample} correspond to the tangent map of Rota--Baxter operators at identity in Example \ref{groupexample}.
\end{remark}

Let \((\mathfrak{g}, [\cdot, \cdot]_{\mathfrak{g}})\) be a Lie algebra and \(\Der(\mathfrak{g})\) denote the Lie algebra of derivations of \(\mathfrak{g}\). Recall that \(\ad: \mathfrak{g} \rightarrow \Der(\mathfrak{g})\) is defined by \(\ad_x(y) = [x, y]\) for \(x, y \in \mathfrak{g}\), representing the adjoint action. Using the adjoint action, we can construct the semi-direct product \(\mathfrak{g} \ltimes_{\ad} \mathfrak{g}\) of Lie algebras, where the Lie bracket is defined by
$$[(x_1, y_1), (x_2, y_2)]_{\mathfrak{g} \ltimes_{\ad}   \mathfrak{g}}=\big([(x_1, x_2)]_{\mathfrak{g}}, [x_1, y_2]_{\mathfrak{g}}+[y_1, x_2]_{\mathfrak{g}}+ [y_1, x_1]_{\mathfrak{g}}\big),$$
for all $x_1, x_2, y_1, y_2 \in \mathfrak{g}.$

Similar to Rota--Baxter operators on groups, we describe the relationship between Lie subalgebras $\mathfrak{g} \ltimes_{\ad} \mathfrak{g}$ and Rota--Baxter operators on $\mathfrak{g}$. For further details, we suggest referring to \cite[Proposition 2.5]{JYC22}.

\begin{prop}\label{liegraph}
Let $(\mathfrak{g}, [\cdot, \cdot])$ be a Lie algebra. Then Rota--Baxter operators on $\mathfrak{g}$ are in one-to-one correspondence with Lie subalgebras of  $\mathfrak{h}$ of $\mathfrak{g} \ltimes_{\ad} \mathfrak{g}$ such that the projection of the second coordinate from $\mathfrak{h}$ onto $\mathfrak{g}$ is  a bijection.
\end{prop}

\begin{proof}
The proof follows along the same lines as Proposition \ref{grptorb}.
\end{proof}

\section{Lie groups and exponential map}
In this section, we revisit and establish some general facts about Lie groups and their exponential maps. Primarily, we review basic properties of the Heisenberg Group and Heisenberg Lie algebra. For more details, we refer the reader to \cite{BRCH15}.

\begin{prop}\label{semiiso}
Let $G$ be any group and $\mathfrak{g}$ be  any Lie algebra . Then the following holds:
	\begin{enumerate}
	\item $G \ltimes_{\Ad} G$ is isomorphic to $G \times G$
	\item $\mathfrak{g}  \ltimes_{\ad} \mathfrak{g}$ is isomorphic to $\mathfrak{g} \times \mathfrak{g}.$
	\end{enumerate}
\end{prop}
\begin{proof}
\begin{enumerate}
\item The map $\Phi: G \ltimes_{\Ad} G \rightarrow G \times G$, defined by $(h,g) \mapsto (gh, h)$ is the recquired isomorphism.
\item Similarly, the map $\phi:\mathfrak{g}  \ltimes_{\ad} \mathfrak{g} \rightarrow \mathfrak{g} \times \mathfrak{g}$, defined by $(x,y) \mapsto (y+x, x)$ is the recquired isomorphism.
\end{enumerate}
\end{proof}

\begin{remark}
Let  $\mathfrak{g}$ be the Lie algebra of Lie group $G$. It is well known that the Lie algebra of the Lie group $G \ltimes_{\Ad} G$ is $\mathfrak{g} \ltimes_{\ad} \mathfrak{g}$, and the Lie algebra of the Lie group $G \times G$ is $\mathfrak{g} \times \mathfrak{g}$. It is easy to see that the map $\phi$, defined in Proposition \ref{semiiso}, is the tangent map of $\Phi$ at the identity. 
\end{remark}

Let $G$ be a Lie group, and let $\mathfrak{g}= T_{e}G$ be the Lie algebra of $G$. Denote $\e_G: \mathfrak{g} \rightarrow G$ as the associated exponential map. Then, we know that the map $\exp_G: \mathfrak{g} \times \mathfrak{g} \rightarrow G \times G,$ defined by $\exp(x,y)=(e_G^x, e_G^y)$, is the exponential map of the Lie group $G \times G$. If $\Phi$ and $\phi$ be the maps defined in Proposition \ref{semiiso}. Then, the map $\Phi^{-1} \exp \phi$ is a map from $\mathfrak{g} \ltimes_{\ad} \mathfrak{g}$ to $G \ltimes_{\Ad} G$, explicitly given by $(x,y) \mapsto (e_G^x, e_G^{x+y}e_G^{-x})$ for all $x, y \in \mathfrak{g}.$ Next, we will prove that this map serves as the exponential map for the Lie group $G \ltimes_{\Ad} G$.

\begin{thm}\label{semiexp}
The exponential map, denoted by \(\Exp_G\), from the tangent space \(\mathfrak{g} \ltimes_{\ad} \mathfrak{g}\) of the Lie group \(G \ltimes_{\Ad} G\), is defined as \(\Exp(x, y) = (e_G^x, e_G^{x+y}e_G^{-x})\) for all \(x, y \in \mathfrak{g}\).
\end{thm}
\begin{proof}
Straight forward calculations shows that for $x,y\in \mathfrak{g}$ and $t,s \in \mathbb{C}$, we have 
\begin{align}
\Exp_G((t+s)(x,y))=&\;\Exp_G(t(x,y)) \Exp_G(s(x,y))\label{firstproperty}\\
\Exp_G\big((x_1,y_1)+(x_2,y_2)\big)=& \;\Exp_G(x_1,y_1) \Exp_G(x_2,y_2), \mbox{ if } [(x_1,y_1),(x_2,y_2)]=0. \nonumber
\end{align}
	
For a given $x,y \in \mathfrak{g}$, define $\zeta: \mathbb{C} \rightarrow  G \ltimes_{\Ad} G$ by 
$$ \zeta(t)=(e_G^{tx},  e_G^{t(x+y)}e_G^{-tx}) $$
It follows from \eqref{firstproperty}  that $\zeta$ is morphism of Lie groups and further we have
\begin{align*}
\dot{\zeta}(t)=\frac{d \big(\zeta(t)\big)}{dt}=& (x e_G^{tx},  (x+y)e_G^{t(x+y)}e_G^{-tx}-x e_G^{t(x+y)} \e^{-tx}).
\end{align*}
Putting $t=0$, in $\dot{\zeta}(t)$, we have $\dot{\zeta}(0)=(x,y).$ This shows that $\zeta$ is one-parameter subgroup of \(G \ltimes_{\Ad} G\) whose tangent vector at identity is equal to $(x,y).$

\end{proof}

\subsection{Heisenberg group and Lie algebra}	

Recall that the Heisenberg group $H$ of $3 \times 3$ matrices over $\mathbb{C}$ is a matrix Lie group containing  matrices of the form
\[ \left[
\begin{matrix}
	1 & a & c \\
	0 & 1 & b \\
	0 & 0 & 1 \\
\end{matrix}
\right]
\]
and corresponding Heisenberg  Lie algebra $\mathfrak{h}$  contains  matrices of the form 
\[ \left[
\begin{matrix}
	0 & a & c \\
	0 & 0 & b \\
	0 & 0 & 0 \\
\end{matrix}
\right]
\]
where $a, b, c \in \mathbb{C}$. 

The following facts are well-known 

\begin{enumerate}
\item  The elements
$$X=\left[ \begin{matrix}
	0 & 1 & 0 \\
	0 & 0 & 0 \\
	0 & 0 & 0 \\
\end{matrix} \right],  
\hspace{.1cm}
Y= \left[\begin{matrix}
	0 & 0 & 0 \\
	0 & 0 & 1 \\
	0 & 0 & 0 \\
\end{matrix}\right],
\hspace{.1cm}
Z=\left[\begin{matrix}
	0 & 0 & 1 \\
	0 & 0 & 0 \\
	0 & 0 & 0 \\
\end{matrix}\right]
$$
constitute a basis of $\mathfrak{h}$ and satisfy the following 
\begin{align}\label{basis}
[X,Y]=Z, \hspace{.2cm} [X,Z]=[Y,Z]=0.
\end{align}

\item  The exponential map $\e_H: \mathfrak{h} \rightarrow H$ is given by 
\begin{align}\label{expmapexpression}
\e_H^x=\I+x+\frac{x^2}{2},
\end{align}
for all  $x \in \mathfrak{h}$, where $\I$ represent the $3 \times 3$ identity matrix. More precisely, if 
$x= \left[ \begin{matrix}
	0 & a & c \\
	0 & 0 & b \\
	0 & 0 & 0 \\
\end{matrix} \right]$ then $\e_H^x= \left[ \begin{matrix}
1 & a & c+ab/2 \\
0 & 1 & b \\
0 & 0 & 1 \\
\end{matrix}\right]
.$

\item The exponential map \(\exp_H\) is a diffeomorphism, and for all \(x, y \in \mathfrak{h}\), it satisfies
\begin{align}\label{propertyexp}
\e_H^x \e_H^y= \e^{x+y+\frac{[x,y]}{2}}_H.
\end{align}

\end{enumerate}

\begin{remark}\label{simpleexp}
Using Theorem\ref{semiexp}, the exponential map $\Exp_H: \mathfrak{h} \ltimes_{\ad} \mathfrak{h} \rightarrow H \ltimes_{\Ad} H$ is given by $\Exp_H(x,y)=\big(\e^x_H, \e_H^{y+\frac{[x,y]}{2}} \big)$ for all $x,y \in \mathfrak{h}$.
\end{remark}

\section{Rota--Baxter operators over Heisenberg  Lie algebra}
In this section, we revisit the classification of Rota–Baxter operators on the Heisenberg Lie algebra as presented in \cite[Theorem 3.2]{JH18}. We represent these operators in matrix form and classify them up to conjugation by an automorphism of \( \mathfrak{h} \). Our matrix representation of these operators is slightly different from the original.

Let us fix an ordered basis \(\{X, Y, Z\}\) for \(\mathfrak{h}\). Then every Rota--Baxter operator $R$ over $\mathfrak{h}$ can be represented by a $3 \times 3$ matrix with entries in $\mathbb{C}$, given by
\[
R= \left[\begin{matrix}
	 r_{11} & r_{12} & r_{13} \\
	r_{21} & r_{22} & r_{23} \\
	r_{31} & r_{32} & r_{33} \\
\end{matrix} \right]
\]
where $r_{ij} \in \mathbb{C}$ for all $1 \leq i,j  \leq 3$, such that
\begin{align}\label{RBop}
R(X) =\; &  r_{11}X+r_{12}Y+r_{13}Z,\\
R(Y) =\; &  r_{21}X+r_{22}Y+r_{23}Z,\\
R(Z)= \; &  r_{31}X+r_{32}Y+r_{33}Z.
\end{align}

\begin{thm}\label{allrbo}
Let $r_{ij} \in \mathbb{C}$ for all $1 \leq i,j \leq 3$. All Rota-Baxter operators on  3-dimensional Heisenberg Lie algebra $\mathfrak{h}$ are the following:

\begin{enumerate}
\item $R_1=\left[ \begin{matrix}
	r_{11} & r_{12} & r_{13} \\
	r_{21} & \frac {r_{33} r_{11}+r_{33}+r_{21} r_{12}}{r_{11}-r_{33}} & r_{23} \\
	0 &0 & r_{33} \\
\end{matrix} \right]$ where $r_{11}-r_{33} \neq 0.$

\item$ R_2=\left[\begin{matrix}
	r_{33} & \frac {-(r_{33}+r^2_{33})}{r_{21}} & r_{13} \\
r_{21} & r_{22} & r_{23} \\
	0 &0 & r_{33} \\
\end{matrix} \right]$ where $r_{21} \neq 0.$

\item$ R_3=\left[\begin{matrix}
-1 & r_{12} & r_{13} \\
	0 & r_{22} & r_{23} \\
	0 &0 & -1 \\
\end{matrix} \right].$ 

\item $ R_4=\left[\begin{matrix}
	0 & r_{12}  & r_{13} \\
	0 & r_{22} & r_{23} \\
	0 &0 & 0 \\
\end{matrix} \right].$ 
\end{enumerate}
\end{thm}

\begin{remark}
Please note that the definition of Rota–Baxter operators of weight $1$ given in \cite{JH18} is actually the definition of Rota–Baxter operators of weight \(-1\) according to our terminology. Therefore, the classification given there pertains to Rota–Baxter operators of weight $-1$. However, we convert these to Rota–Baxter operators of weight 1 using the fact that \(R\) is a Rota–Baxter operator of weight \(-1\) if and only if \(-R\) is a Rota–Baxter operator of weight $1$.
\end{remark}

Next, we present more simplified forms of these Rota–Baxter operators using the classification of automorphisms of \(\mathfrak{h}\), provided in \cite[Proposition 2.2]{WY2015}.

\begin{prop}
	Let $a_{ij} \in \mathbb{C}$ for all $1 \leq i,j \leq 3$. The automorphism group of $\mathfrak{h}$ is
	$$\left\{ \left[ \begin{matrix}
		a_{11} & a_{22}  & a_{13} \\
		a_{21} & a_{22} & a_{23} \\
		0 &0 & a_{11} a_{22}-a_{21} a_{22} \\
	\end{matrix}\right] \hspace{.1cm} \raisebox{-0.3cm}{\rule{0.4pt}{.8cm}} \hspace{.1cm}a_{11} a_{22}-a_{21} a_{22} \neq 0  \right\}.$$
\end{prop}

\begin{thm}\label{eqrbo}
Let  $R= \left[
\begin{matrix}
	r_{11} & r_{12} & r_{13} \\
	r_{21} & r_{22} & r_{23} \\
	0 & 0 & r_{33} \\
\end{matrix}
\right]$
be an arbitrary Rota–Baxter operator on \( \mathfrak{h} \). Upto Jordan canonical forms of  $\left[\begin{matrix}
	r_{11} & r_{12} \\
	r_{21} & r_{22} \\
\end{matrix} \right],$ we have the following Rota--Baxter operators on $\mathfrak{h}$:
\begin{enumerate}
\item  $P_1=\left[ \begin{matrix}
	\frac {(1+r_{22})r_{33}}{r_{22}-r_{33}} & 0 & r_{13} \\
	0 & r_{22} & r_{23} \\
	0 &0 & r_{33} \\
\end{matrix} \right]$ where $r_{22} \neq r_{33}.$

\item $P_2=\left[ \begin{matrix}
	r_{11}& 0 & r_{13} \\
	0 & -1 & r_{23} \\
	0 &0 & -1 \\
\end{matrix} \right].$ 

\item  $P_3=\left[ \begin{matrix}
	r_{11}& 0 & r_{13} \\
	0 & 0 & r_{23} \\
	0 &0 & 0 \\
\end{matrix} \right].$ 

\item  $P_4=\left[ \begin{matrix}
	r_{11} & 1 & r_{13} \\
	0 & r_{11} & r_{23} \\
	0 &0 &\frac{r^2_{11}}{1+2r_{11}} \\
\end{matrix} \right]$ where $r_{11} \neq -\frac{1}{2}$.
\end{enumerate}

\end{thm}

\begin{proof}
	Define $
	M_R= \left[\begin{matrix}
		r_{11} & r_{12} \\
		r_{21} & r_{22} \\
	\end{matrix} \right].
	$
	With the help of \( \Aut(\mathfrak{h}) \), we can transform the matrix \( M_R \) into its Jordan canonical form \( J \). We know that the possible Jordan canonical forms of $M_R$ are :
	
	\begin{enumerate}
\item $\left[\begin{matrix}
	r_{11} & 0 \\
	0 & r_{22} \\
\end{matrix} \right],$ 

\vspace{.2cm}
\item $\left[\begin{matrix}
	r_{11} & 1 \\
	0 & r_{11} \\
\end{matrix} \right]$.
	\end{enumerate}
	
	Let $S_1=\left[ \begin{matrix}
		r_{11} & 0 & r_{13} \\
	0 & r_{22} & r_{23} \\
		0 &0 & r_{33} \\
	\end{matrix} \right].$ By skew-symmetry and using that $Z$ is central, in order to show that $S_1$ is a Rota--Baxter operator, we only need to check 
\begin{align}
[S_1(X), S_1(Y)]=\; & S_1\big([S_1(X), Y]+[X, S_1(Y)]+[X,Y]  \big) \label{RBO1}
\end{align}
Using \eqref{basis} and \eqref{RBop} in \eqref{RBO1}, we get
\begin{align}\label{firstLHS}
[S_1(X), S_1(Y)]= \; &[r_{11}X+r_{13}Z, r_{22}Y+r_{23}Z] \nonumber\\
=& \; r_{11} r_{22}Z
\end{align} 
while,
\begin{align}\label{firstRHS}
S_1\big([S_1(X), Y]+[X, S_1(Y)]+[X,Y]  \big)=\; &R\big( [r_{11}X+r_{13}Z, Y]+[X, r_{22}Y+r_{23}Z]+ Z\big) \nonumber\\
=\; & R\big( r_{11}Z+ r_{22}Z+Z)\nonumber\\
=\; & (r_{11}+r_{22}+1)r_{33}Z.
\end{align} 
Comparing the coefficients in \eqref{firstLHS} and \eqref{firstRHS}, we have
\begin{align}\label{firsteq}
(r_{11}+r_{22}+1)r_{33}=\; & r_{11} r_{22}.
\end{align}
.
 Using Mathematica \cite{math}, we solved \eqref{firsteq}, and the only possible solutions are as follows:
\begin{enumerate}
\item $r_{11}=\frac {(1+r_{22})r_{33}}{r_{22}-r_{33}},$ $r_{22} \neq r_{33},$
\item $r_{22}=-1,$   and  $r_{33}=-1,$ 
\item  $r_{22}=0$ and  $r_{33}=0.$ 
\end{enumerate}

Let $S_2=\left[ \begin{matrix}
	r_{11} & 1 & r_{13} \\
	0 & r_{11} & r_{23} \\
	0 &0 & r_{33} \\
\end{matrix} \right].$ Using the similar method, we have 
\begin{align}\label{SecondLHS}
	[S_2(X), S_2(Y)]= \; &[r_{11}X+Y+r_{13}Z, r_{11}Y+r_{23}Z] \nonumber\\
	=& \; r_{11}^2Z,
\end{align} 
and 
\begin{align}\label{SecondRHS}
S_2\big([S_2(X), Y]+[X, S_2(Y)]+[X,Y]  \big)=\; &S_2(r_{11}Z+r_{11}Z+Z) \nonumber\\
=&\; (2r_{11}+1)r_{33}.
\end{align}
Comparing \eqref{SecondLHS} and \eqref{SecondRHS}, we have 
$$r^2_{11}=(2r_{11}+1)r_{33}.$$
The only possible solution is $r_{33}=\frac{r^2_{11}}{1+2r_{11}}$ such that $r_{11} \neq -\frac{1}{2}.$ This completes the proof.
	
\end{proof}

\section{Rota--Baxter operators over Heisenberg  group}

In this section, we induce the Rota--Baxter operators identified for the Heisenberg Lie algebra onto the Heisenberg Lie group.

By using Theorem \ref{allrbo}, we can conclude that any Rota--Baxter operator over $\mathfrak{h}$ has the form $$R= \left[ \begin{matrix}
	r_{11} & r_{12} & r_{13} \\
	r_{21} & r_{22} & r_{23} \\
0 & 0 & r_{33} \\
\end{matrix} \right].
$$
Let $\text{Gr}(R)=(R(x), x)$ be the graph of $R$. From Proposition \ref{liegraph}, we know that $\text{Gr}(R)$ is a Lie subalgebra of $\mathfrak{g} \ltimes_{\ad} \mathfrak{g}$. According to the theory of simply connected Lie groups, there exists a unique subgroup of $G \ltimes_{\Ad} G$ such that its tangent space at the identity is $\text{Gr}(R)$. More precisely, this subgroup is generated by $\Exp_H(\text{Gr}(R))$, where $\Exp_H$ is the exponential map of the semi-direct product as defined in Proposition \ref{semiexp}.

\begin{prop}
The image of \(\text{Gr}(R)\) under the map \(\Exp_H\) forms a group.
\end{prop}
\begin{proof}
Let $x,y \in \mathfrak{g},$ then by Theorem \ref{semiexp}, we know that $\Exp_H(R(x),x)=(e^{R(x)}_H, e^{R(x)+x}_H e_H^{-R(x)})$ and $\Exp_H(R(y),y)=(e^{R(y)}_H, e^{R(y)+y}_H e_H^{-R(y)}).$ We have
\begin{align}
(&\e_H^{R(x)}, \e^{R(x)+x}_H \e_H^{-R(x)}) (\e^{R(y)}_H, \e^{R(y)+y}_H \e_H^{-R(y)})\nonumber\\
& =\;  (\e_H^{R(x)}\e^{R(y)}_H, \e^{R(x)+x}_H \e_H^{-R(x)} \e_H^{R(x)} \e^{R(y)+y}_H \e_H^{-R(y)} \e_H^{-R(x)} ),\nonumber\\
& =\;  (\e^{R(x)+R(y)+ \frac{[R(x),R(y)]}{2}}_H,  \e^{R(x)+x}_H  \e^{R(y)+y}_H \e_H^{-R(y)} \e_H^{-R(x)}), \text{ using \eqref{propertyexp} } \nonumber\\
& =\;  (\e^{R(x)+R(y)+ \frac{[R(x),R(y)]}{2}}_H,  \e^{R(x)+x+R(y)+y +\frac{[x+R(x),y+R(y)]}{2}}_H \e_H^{-R(y)-R(x)-  \frac{[R(x),R(y)]}{2} }),\nonumber \\
& =\; \Big(\e^{R(x+y+ \frac{[R(x),y]+[x, R(y)] +[x,y]}{2})}_H, \e_H^{(x+y+ \frac{[R(x),y]+[x, R(y)] +[x,y]}{2})+ (R(x)+R(y)+\frac{[R(x),R(y)]}{2})} \e_H^{-R(y)-R(x)-  \frac{[R(x),R(y)]}{2} } \Big).\label{graphgroup}
\end{align}
Using $z=x+y+ \frac{[R(x),y]+[x, R(y)] +[x,y]}{2}$  in \eqref{graphgroup}, we have 
$$(\e_H^{R(x)}, \e^{R(x)+x}_H \e_H^{-R(x)}) (\e^{R(y)}_H, \e^{R(y)+y}_H \e_H^{-R(y)})=(\e_H^{R(z)}, \e_H^{z+R(z)} \e^{-R(z)}_H).$$
This demonstrates that the set $\Exp_H(\text{Gr}(R))$ is closed under multiplication. Similarly, it can be observed that $\Exp_H(R(x),x)^{-1} = \Exp_H(-R(x),-x)$, indicating that $\Exp_H(\text{Gr}(R))$ is also closed under taking inverses. Consequently, this establishes that $\Exp_H(\text{Gr}(R))$is  a group.
\end{proof}

\begin{thm}\label{RBmatrix}
Let $R=\left[\begin{matrix}
	r_{11} & r_{12} & r_{13} \\
	r_{21} & r_{22} & r_{23} \\
	0 & 0 & r_{33} \
\end{matrix} \right]$ be any Rota--Baxter operator over $\mathfrak{h}$. Then the map $x \mapsto x+\frac{ [R(x),x]}{2}$ is a bijection of $\mathfrak{h}$.
\end{thm}

\begin{proof}
Let $P: \mathfrak{h} \rightarrow \mathfrak{h}$ be defined by $P(x):=x+\frac{ [R(x),x]}{2}$ for all $ x \in \mathfrak{h}$. Let $x=aX+bY+cZ$ for some $a,b,c \in \mathbb{C}$. Direct calculations shows that 
\begin{align*}
R(x)=(ar_{11}+br_{21})X+(ar_{12}+br_{22})Y+(ar_{13}+br_{23}+c r_{33})Z
\end{align*}
and 
\begin{align*}
[R(x),x]=\;& [(ar_{11}+br_{21})X+(ar_{12}+br_{22})Y+(ar_{13}+br_{23}+c r_{33})Z, aX+bY+cZ ],\nonumber \\
=\;&\big( b(ar_{11}+br_{21})-a(ar_{12}+br_{22})\big)Z. 
\end{align*}

Hence, we have  
\begin{align}\label{RBeqn}
P(x)=aX+bY +\frac{1}{2}\big(2c+b(ar_{11}+br_{21})-a(ar_{12}+br_{22}) \big)Z.
\end{align}

Similarly, let  $y=a^{\prime}X+b^{\prime}Y+c^{\prime}Z$ for some $a^{\prime}, b^{\prime}, c^{\prime} \in \mathbb{C}$. We have
$$P(y)=a^{\prime}X+b^{\prime}Y +\frac{1}{2}\big(2c^{\prime}+b^{\prime}(a^{\prime} r_{11}+b^{\prime}r_{21})-a^{\prime}(a^{\prime}r_{12}+b^{\prime}r_{22}) \big)Z.$$
It can be easily seen that $P(x)=P(y)$ implies $x=y$. This shows that $P$ is injective.

To demonstrate that $P$ is bijective, consider any $x = aX + bY + cZ \in \mathfrak{g}$ and let $$y = aX + bY + \big(c -\frac{1}{2} (b(ar_{11} + br_{21}) + a(ar_{12} + br_{22}))\big)Z.$$ Then, using \eqref{RBeqn}, we have $P(y) = x$, indicating that $P$ is surjective. Thus, $P$ is a bijection.
\end{proof}

\begin{cor}\label{bijection}
For any Rota--Baxter operator $R$ on $\mathfrak{h}$, the projection of the second coordinate from $\Exp_H(\text{Gr}(R))$ onto $H$ is a bijection.
\end{cor}
\begin{proof}
Let $\pi_2$ denote the projection of the second coordinate from $\Exp_H(\text{Gr}(R))$ onto $H$. Then $\pi_2(\e^{R(x)}_H, \e^{R(x)+x}_H \e_H^{-R(x)})=\e^{R(x)+x}_H \e_H^{-R(x)}.$ Now, using Remark \ref{simpleexp}, we have $$\pi_2(\e^{R(x)}_H, \e^{R(x)+x}_H \e_H^{-R(x)})=\e_H^{x+\frac{[R(x),x]}{2}}.$$ Using the fact that $\e_H$ and $P$ are bijections, we can conclude that $\pi_2$ is a bijection.
\end{proof}

\begin{cor}\label{algebratogroup}
Every Rota--Baxter operator on $\mathfrak{h}$ induces a Rota--Baxter operator on $H$.
\end{cor}

\begin{proof}
Let $R$ be any Rota--Baxter operator on $\mathfrak{h}$. Utilizing the fact that $\e_H$ is a bijection on $H$, along with Corollary \ref{bijection} and Proposition \ref{graphgroup}, we can establish that for $x \in  \mathfrak{h}$, the map $\mathcal{R}: H \rightarrow H$, defined as
$$\mathcal{R}(e_H^{P(x)})=e_H^{R(x)}$$ is a Rota--Baxter operator on $H$,  where  $P(x)=x+\frac{ [R(x),x]}{2}$.
\end{proof}
The following question arises naturally:
\begin{question}
Do equivalent Rota–Baxter operators on \( \mathfrak{h} \) induce equivalent Rota–Baxter operators on \( H \)?
\end{question}

 Next we see more precise form of Rota--Baxter operator over $H$. Consider 
 $$x=\left[\begin{matrix}
	0 & a & c \\
	0 & 0 & b \\
	0 & 0 & 0 \\
\end{matrix}\right] \mbox{ and }  R=\left[\begin{matrix}
r_{11} & r_{12} & r_{13} \\
r_{21} & r_{22} & r_{23} \\
0 & 0 & r_{33} \
\end{matrix} \right].$$ 

Through direct calculations and using Theorem \ref{RBmatrix} \eqref{expmapexpression}, we have
\begin{enumerate}

\item $R(x)= \left[\begin{matrix}
0 & (ar_{11}+br_{21}) & (ar_{13}+br_{23}+c r_{33}) \\
0 & 0 & (ar_{12}+br_{22}) \\
0 & 0 &0 \
\end{matrix} \right]$ 
\item $P(x)=x+\frac{ [R(x),x]}{2}=\left[\begin{matrix}
	0 & a & \frac{1}{2}\big(2c+b(ar_{11}+br_{21})-a(ar_{12}+br_{22}) \big) \\
	0 & 0 & b \\
	0 & 0 & 0 \\
\end{matrix}\right]$

\item $
\e_H^{P(x)}=\left[\begin{matrix}
	1 & a & \frac{1}{2} (ab + 2c + abr_{11} - a^2r_{12} + b^2r_{21} - abr_{22}) \\
	0 & 1 & b \\
	0 & 0 & 1 \\
\end{matrix}\right]
$

\item $\e_H^{R(x)}=\left[\begin{matrix}
	1 & a r_{11} + b r_{21} & \frac{1}{2} (a^2 r_{11} r_{12} + 2 a r_{13} + a b r_{12} r_{21} + a b r_{11} r_{22} + b^2 r_{21} r_{22} + 2 b r_{23} + 2 c r_{33}) \\
	0 & 1 & a r_{12} + b r_{22} \\
	0 & 0 & 1
\end{matrix}\right]$.
\end{enumerate}
It is clear from Corollary \ref{bijection} that the mapping $\e_H^{P(x)} \mapsto \e_H^{R(x)}$ defines a Rota--Baxter operator on $H$, yet the description of the matrix for $\e_H^{P(x)} \in H$ is somewhat complicated. We will simplify this as follows.

Let $y=\left[\begin{matrix}
	0 & a & c-\frac{1}{2}\big(ab-abr_{22}+abr_{11}+b^2r_{21}-a^2r_{12}\big) \\
	0 & 0 & b \\
	0 & 0 & 0 \\
\end{matrix}\right]$. Then we have

\begin{enumerate}
\item $R(y)=\left[\begin{matrix}
	0 & (ar_{11}+br_{21}) & ar_{13}+br_{23}+\Big(c-\frac{1}{2}\big( ab-abr_{22}+abr_{11}+b^2r_{21}-a^2r_{12}\big)\Big) r_{33} \\
	0 & 0 & (ar_{12}+br_{22}) \\
	0 & 0 &0 \
\end{matrix}\right].$
\vspace{.2cm}
\item $P(y)=y+\frac{ [R(y),y]}{2}=
\left[\begin{matrix}
	0 & a & \frac{1}{2} \left( b (a r_{11} + b r_{21}) - a (a r_{12} + b r_{22}) \right. \\
	&& \quad \left. + 2 \left( c + \frac{1}{2} \left( -a b - a b r_{11} + a^2 r_{12} - b^2 r_{21} + a b r_{22} \right) \right) \right) \\
	0 & 0 & b \\
	0 & 0 & 0
\end{matrix}\right].
$
\vspace{.2cm}
\item  $\e_H^{P(y)}=\left[\begin{matrix}
	1 & a & c \\
	0 & 1 & b \\
	0 & 0 &1 \
\end{matrix}\right].$
\vspace{.2cm}

\item $\e_H^{R(y)}=\left[\begin{matrix}
	1 & a r_{11} + b r_{21} & \frac{1}{2} \Bigl(2 a r_{13} + 2 b r_{23} + b^2 r_{21} (r_{22} - r_{33}) + 2 c r_{33} + a^2 r_{12} (r_{11} + r_{33}) \\
	& & \quad + a b (r_{12} r_{21} + r_{11} (r_{22} - r_{33}) + (-1 + r_{22}) r_{33})\Bigr) \\
	0 & 1 & a r_{12} + b r_{22} \\
	0 & 0 & 1
\end{matrix}\right].$
\end{enumerate}

Hence, we now have a more standard form of Corollary \ref{algebratogroup}, as follows
\begin{thm}\label{RBOHLG}
Let  $R=\left[\begin{matrix}
	r_{11} & r_{12} & r_{13} \\
	r_{21} & r_{22} & r_{23} \\
	0 & 0 & r_{33} \
\end{matrix}\right]$ be any Rota--Baxter operator over $\mathfrak{h}$. Then it induces a Rota--Baxter operator over $H$, given by the mapping 
$$ \left[\begin{matrix}
	1 & a & c \\
	0 & 1 & b \\
	0 & 0 &1 \
\end{matrix} \right] \mapsto \left[\begin{matrix}
1 & a r_{11} + b r_{21} & \frac{1}{2} \Bigl(2 a r_{13} + 2 b r_{23} + b^2 r_{21} (r_{22} - r_{33}) + 2 c r_{33} + a^2 r_{12} (r_{11} + r_{33}) \\
& & \quad + a b (r_{12} r_{21} + r_{11} (r_{22} - r_{33}) + (-1 + r_{22}) r_{33})\Bigr) \\
0 & 1 & a r_{12} + b r_{22} \\
0 & 0 & 1
\end{matrix}\right].$$
\end{thm}

Next, we exhibit all skew left brace structures induced by these Rota--Baxter operators on $H$, as given by the formula in Theorem \ref{VVSB}.

Let  $R=\left[\begin{matrix}
	r_{11} & r_{12} & r_{13} \\
r_{21} & r_{22} & r_{23} \\
	0 &0 & r_{33} \\
\end{matrix}\right]$   and 
$y_1 =\left[ \begin{matrix}
	0 & a_1 & c_1-\frac{1}{2}\big(a_1b_1-a_1b_1r_{22}+a_1b_1r_{11}+b_1^2r_{21}-a_1^2r_{12}\big) \\
	0 & 0 & b_1 \\
	0 & 0 & 0 \\
\end{matrix}\right],$

$y_2=\left[ \begin{matrix}
	0 & a_2 & c_2 - \frac{1}{2}\big(a_2b_2 - a_2b_2 r_{22} + a_2 b_2 r_{11} + b_2^2 r_{21} - a_2^2 r_{12}\big) \\
	0 & 0 & b_2 \\
	0 & 0 & 0 \\
\end{matrix}\right].$
The multiplication operation of the descendant group of $\mathcal{R}$ is given by the following formula
$$\e_H^{P(y_1)} \circ_{\mathcal{R}} \e_H^{P(y_2)} = \e_H^{P(y_1)} \e_H^{R(y_1)} \e_H^{P(y_2)}\e_H^{-R(y_1)}.$$

Direct calculations using Mathematica \cite{math}, shows that   $\e_H^{P(y_1)} \e_H^{R(y_1)} \e_H^{P(y_2)}\e_H^{-R(y_1)}$ is equal to 
$$
\left[\begin{matrix}
	1 & a_1+a_2  & c_1+c_2 +a_1\big(b_2(1+r_{11})-a_2 r_{12} \big)+b_1 \big(b_2r_{21}-a_2r_{22}\big) \\
	
	0 & 1 & b_1+b_2 \\
	0 & 0 & 1
\end{matrix}\right].$$

\begin{remark}
 It is easy to see that if $M_R= \left[\begin{matrix}
	r_{11} & r_{12} \\
	r_{21} & r_{22}\\
\end{matrix}\right]$ is a null  matrix, then the skew left brace structure induced by $R$ is trivial. Please note that the descendant group $\mathcal{R}$ depends solely on how $R$ acts on $X$ and $Y$ in $\mathfrak{h}$.
\end{remark}

\begin{thm}\label{main}
Let \( P_i \) for \( 1 \leq i \leq 4 \) be the Rota–Baxter operators on $\mathfrak{h}$ as classified in Theorem \ref{eqrbo}, and let $\mathcal{P}_i$ be the Rota–Baxter operators induced by \( P_i \) on \( H \). Then, for \( a_j, b_j, \) and \( c_j \)  $( 1 \leq j \leq 2 )$ in \( \mathbb{C} \), the descendant group of $\mathcal{P}_i$ is given by:

$$ \left[\begin{matrix}
	1 & a_1 & c_1 \\
	0 & 1 & b_1 \\
	0 & 0 & 1 \
\end{matrix}\right] \circ_{\mathcal{P}_i}\left[ \begin{matrix}
1 & a_2 & c_2 \\
0 & 1 & b_2 \\
0 & 0 & 1 \
\end{matrix}\right]=\left[\begin{matrix}
1 & a_1+a_2 & q_i \\
0 & 1 & b_1+b_2 \\
0 & 0 &1 \
\end{matrix}\right],
$$
where,
\begin{align*}
q_1=\; &c_1+c_2 +a_1\big(b_2(1+\frac {(1+r_{22})r_{33}}{r_{22}-r_{33}})\big)-b_1 a_2r_{22},\\
q_2=\;& c_1+c_2 +a_1b_2(1+r_{11}) +b_1 a_2,\\
q_3=\;& c_1+c_2 +a_1b_2(1+r_{11}),\\
q_4=\; & c_1+c_2 +a_1\big(b_2(1+r_{11})-a_2  \big)-b_1 a_2r_{11}.
\end{align*}
\end{thm}

\begin{ack}
	{\rm The author is very thankful to the anonymous referee for suggesting Theorem \ref{eqrbo} and for many helpful suggestions that significantly enhanced the presentation of this article. The author also expresses gratitude to IISER Mohali for providing the institute's postdoctoral fellowship and to Anoop Singh for his invaluable assistance in proving Theorem \ref{semiexp}.}
\end{ack}


\begin{thebibliography}{99}
	\bibitem{VV2022}
V.~Bardakov and V.~Gubarev, \textit{Rota--Baxter groups, skew left braces, and the Yang-Baxter equation},  J. Algebra \textbf{ 587} (2022), 328--351.

	\bibitem{VV2023}
V.G.~Bardakov and V.~Gubarev, Rota–Baxter operators on groups, {\it Proc Math. Sci. } \textbf{133}, 4 (2023).



\bibitem{CMP}
F.~ Catino, M.~Mazzotta and P.~Stefanelli, \textit{Rota--Baxter operators on Clifford semigroups and the Yang-Baxter equation}, J.Algebra \textbf{622} (2023), 587-613.

\bibitem{AS22}
A. ~Caranti and L.~Stefanello, \textit{Skew braces from Rota–Baxter operators: a cohomological characterisation and some examples}, Annali di Matematica Pura ed Applicata (1923 -) \textbf{31} (2022).

\bibitem{LHY2021} 
 L.~Guo, H.~Lang and Y.~ Sheng, \textit{Integration and geometrization of Rota--Baxter Lie algebras}, Adv. Math. \textbf{387} (2021), 107834, 34 pp.
 
 
\bibitem{JYC22}
 J.~ Jiang, Y.~ Sheng and C.~ Zhu, \textit{Lie theory and cohomology of relative Rota--Baxter operators},  Journal of the London Mathematical Society, (2024)
 https://doi.org/10.1112/jlms.12863.
 
 \bibitem{JH18}
G.~Ji and X. Hua, \textit{Rota--Baxter operators of 3-dimensional Heisenberg Lie algebra,} Korean J. Math. Vol. \textbf{26} No. 1 (2018),53-60.
 
 \bibitem{GV17}
 L.~Guarnieri and L.~ Vendramin, \textit{Skew braces and the Yang-Baxter equation},  Math. Comp. \textbf{86} (2017),  2519-2534.
 
 \bibitem{BRCH15}
  Brian C.~Hall \textit{Lie groups, Lie algebras, and representations. An elementary introduction,} Graduate Texts in Mathematics, 222, 2nd ed., (2015) Springer.
 
 \bibitem{PBG17}
 J.~ Pei, C.~Bai and L.~Guo \textit{Splitting of Operads and Rota--Baxter Operators on Operads,} Appl Categor Struct \textbf{25}, 505–538 (2017).
 
 \bibitem{PBG14}
 J.~ Pei, C.~Bai and L.~Guo \textit{ Rota--Baxter operators on $sl(2,C)$ and solutions of the classical Yang-Baxter equation}, J. Math. Phys. \textbf{55}, (2014), 021701.
 
 \bibitem{NM1} N.~Rathee and M.~Singh, \textit{Relative Rota--Baxter groups and skew left braces}, Forum Math. (2024), https://doi.org/10.1515/forum-2024-0020
 
 \bibitem{GCR1969}
 G.C~Rota, \textit{Baxter algebras and combinatorial identities}, I, II. 
 {\it Bull. Amer. Math. Soc.} \textbf{75} (1969),  325-329; ibid. 75 1969 330-334.
 

 \bibitem{math}
 Wolfram Research, Inc., \textit{Mathematica}, Version 13.0, Champaign, IL (2024).
 
 https://www.wolfram.com/mathematica.
 
 \bibitem{UK08}
 K.~ Uchino, \textit{Quantum Analogy of Poisson Geometry, Related Dendriform Algebras and Rota–Baxter Operators}, Lett Math Phys  \textbf{85}, 91–109 (2008).
 
 \bibitem{WY2015}
 W.~ Yoo, \textit{The Automorphisms of a Lie algebra,} Applied Mathematical Sciences, \textbf{9} (2015), 121-127.
 
 \bibitem{ZGR19}
S.~ Zheng, L.~ Guo and M.~ Rosenkranz, \textit{ Classification of Rota--Baxter operators on semigroup algebras of order two and three,} Commun Algebra. (2019) \textbf{47(8)},3094–3116.
 
 
	\end{thebibliography}
\end{document}